\documentclass[a4paper,english]{lipics-v2021}
\hideLIPIcs
\nolinenumbers

\usepackage{amsthm}
\usepackage{amsmath}
\usepackage{amssymb}
\usepackage{amsfonts}
\usepackage{epsfig}
\usepackage{graphics}
\usepackage{graphicx}
\usepackage{booktabs}
\usepackage{pdfpages}
\usepackage{MnSymbol}

\usepackage{complexity}
\usepackage{url}

\usepackage{latexsym}
\usepackage{psfrag}
\usepackage{enumerate}
\usepackage{here}
\usepackage{esvect}
\usepackage[update,prepend]{epstopdf}
\usepackage{anyfontsize}
\usepackage{url}
\usepackage{hyperref}

\usepackage{microtype}

\newcommand{\etal}{{et~al.}}
\newcommand{\ie}{{i.e.}}
\newcommand{\eg}{{e.g.}}

\def\L{\mathcal L}

\providecommand{\intd}[0]%
{\;\mbox{d}}

\usepackage[noline,linesnumbered]{algorithm2e}
\SetArgSty{textrm}
\let\oldnl\nl
\newcommand{\nonl}{\renewcommand{\nl}{\let\nl\oldnl}}
\makeatletter
\def\TitleOfAlgo{\@ifnextchar({\@TitleOfAlgoAndComment}{\@TitleOfAlgoNoComment}}
\def\@TitleOfAlgoAndComment(#1)#2{\nonl\hspace*{-1.5em}#2 #1\;}
\def\@TitleOfAlgoNoComment#1{\nonl\hspace*{-1.5em}#1\;}
\makeatother
\DontPrintSemicolon

\newcommand{\later}[1]{}
\newcommand{\old}[1]{}

\title{Note on the Number of Almost Ordinary Triangles}

\titlerunning{Note on the Number of Almost Ordinary Triangles}

\author{Adrian Dumitrescu}
{Algoresearch L.L.C., Milwaukee, WI, USA, and 
Research Institute of the University of Bucharest, Romania, and 
Alfr\'ed R\'enyi Institute of Mathematics, Budapest, Hungary}
{ad.dumitrescu@algoresearch.org}
{0000-0002-1118-0321}
{}

\author{J\'anos Pach}
{Alfr\'ed R\'enyi Institute of Mathematics, Budapest, Hungary}
{pach@renyi.hu}
{0000-0002-2389-2035}
{}

\authorrunning{Adrian Dumitrescu and J\'anos Pach}

\Copyright{Adrian Dumitrescu and J\'anos Pach}

\funding{Work on this paper was supported by ERC Advanced Grant no. 882971  ``GeoScape.''
  Work by the second named author was also supported by NKFIH
  (National Research, Development and Innovation Office) grant K-131529.}



\begin{document}

\maketitle

\begin{abstract}
  Let $X$ be a set of $n$ points in the plane, not all on a line. According to the Gallai-Sylvester
  theorem, $X$ always spans an \emph{ordinary line}, i.e., one that passes through precisely 2
  elements of $X$. Given an integer $c\ge 2,$ a \emph{line} spanned by $X$ is called
  \emph{$c$-ordinary} if it passes through at most $c$ points of $X$. A \emph{triangle} spanned by 3
  noncollinear points of $X$ is called \emph{$c$-ordinary} if all 3 lines determined by its sides
  are $c$-ordinary. Motivated by a question of Erd\H os, Fulek~\etal~\cite{FMN+17}
  proved that there exists an absolute constant $c > 2$ such that if $X$ cannot be covered by 2
  lines, then it determines at least one $c$-ordinary triangle. Moreover, the number of such
  triangles grows at least linearly in $n$. They raised the question whether the true growth rate of
  this function is superlinear. 

  We prove that if $X$ cannot be covered by 2 lines, and no line passes through more than $n-t(n)$
  points of $X$, for some function $t(n)\rightarrow\infty,$ then the number of $17$-ordinary
  triangles spanned by $X$ is at least constant times $n \cdot t(n)$, i.e., superlinear in $n$. We also
  show that the assumption $t(n)\rightarrow\infty$ is necessary. If we further assume that no line
  passes through more than $n/2-t(n)$ points of $X$, then the number of $17$-ordinary triangles
  grows superquadratically in $n$. This statement does not hold if $t(n)$ is bounded. We close this
  paper with some algorithmic results. In particular, we provide a $O(n^{2.372})$ time algorithm for
  counting all $c$-ordinary triangles in an $n$-element point set, for any $c<n$. 
\end{abstract}

\section{Introduction} \label{sec:intro}

Given a finite point set $X$, a connecting line is called \emph{ordinary} if it is incident to precisely
two points of $X$. The question of whether every noncollinear point set determines an ordinary line
was first considered by Sylvester~\cite{Sy893} at the end of the $19$th century. The first proofs of
existence of such a line were found independently by Gallai and by Melchior~\cite{Mel41} in the 1940s.
This result is now commonly referred to as the \emph{Sylvester--Gallai Theorem}:
Every set of $n$ noncollinear points in the plane admits an ordinary line.
Many related problems and statements can be found in~\cite{BM90,BMP05,Chv21,CFG91,EP95,GT13,KW91,PS09}.

\smallskip
Motzkin~\cite{Mo51} was the first to show that the number of ordinary lines tends to infinity with $n$.
Kelly and Moser~\cite{KM58} proved that the dependence is in fact linear:
They showed that $n$ noncollinear points in the plane determine at least $3n/7$ ordinary lines,
which is tight for $n=7$. Csima and Sawyer~\cite{CS93} raised this bound to $6n/13$ for $n \geq 8$. 
Finally, Green and Tao~\cite{GT13} proved that if $n$ is sufficiently large, then there are at least
$n/2$ ordinary lines. On the other hand, there are arbitrarily large points sets with no more than
$n/2$ ordinary lines; see, \eg, \cite[Ch.~7.2]{BMP05}.  

\smallskip 
Given a finite point set $X$ in the plane and a constant $c\ge 2$, a connecting line is called
a $c$-\emph{ordinary  line} if it is incident to at most $c$ points of $X$; see~\cite{BVZ16,FMN+17}. 
Three noncollinear elements of $X$ make a $c$-\emph{ordinary triangle} if each of
the three lines induced by them is $c$-ordinary.

Motivated by a question of Erd\H os (see~\cite{BM90,BVZ16}),
de Zeeuw~\cite{Zee18} asked whether there exists a constant $c \geq 2$ such that every $n$-element point set $X$
in the plane which cannot be covered by two lines has a $c$-ordinary triangle.
Clearly, the latter assumption is necessary, because if all points of $X$ lie on the union of two lines
and each of these lines contains more than $c$~points, then $X$ spans no $c$-ordinary triangle. 
Fulek, Mojarrad, Nasz\'odi, Solymosi, Stich, and Szedl\'ak~\cite{FMN+17} gave a positive answer to de Zeeuw's
question, with $c=12000$. The value of $c$ was reduced to 11 by Dubroff~\cite{Dub18}. On the other hand,
it is known that the statement is not true with $c=2$ (see, \eg, ~\cite[Fig.~4]{GT13})
\smallskip

It follows from the proof of Fulek~\etal~\cite{FMN+17} that the number of $c$-ordinary triangles grows linearly in~$n$.
They raised the question whether, for a suitable constant $c>2$, the true growth rate is superlinear. The next statement
shows that the problem needs to be carefully formulated, otherwise the answer is negative.

\begin{proposition} \label{prop:k1}
  Let $c, k\ge 2$, and $n$ be any integers satisfying $n \geq 2k+c$.

  Then there exists an $n$-element point set $X$ that cannot be covered by $k$ lines
  and such that the number of $c$-ordinary triangles spanned by $X$ is at most $2k^2 n$.
\end{proposition}

In particular, for a fixed $k$, the number of $c$-ordinary triangles spanned by $X$ is at most linear in $n$. However,
if we assume that there is no line that contains all but a constant number of points of $X$, then the number of
$c$-ordinary triangles indeed grows superlinearly in $n$. More precisely, we have the following.

\begin{theorem} \label{thm:n-t}
  Let $t(n) \leq 0.1 n$ be a function of $n$, with $\lim_{n \to \infty} t(n) =\infty$.
  Let $X$ be an $n$-element point set in the plane that cannot be covered by two lines,
  and suppose that $X$ has at most $n-t(n)$ collinear points.

  Then, for $c=17$, the number of $c$-ordinary triangles spanned by $X$ is at least $\Omega(n \cdot t(n))$.
This bound is tight, apart from the multiplicative constant hidden in the $\Omega$-notation.
\end{theorem}

We can substantially improve on the above bound under the assumption that no line contains more than half of the
points. In fact, it follows from our next theorem that if the maximum number of collinear points is smaller than
$(1/2-\varepsilon)n$, for a small $\varepsilon>0$, then a positive fraction of all triangles are $c$-ordinary. 

\begin{theorem} \label{thm:n/2-t}
  Let $t(n) \leq 0.1 n$ be a function of $n$, with $\lim_{n \to \infty} t(n) =\infty$.
  Let $X$ be an $n$-element point set in the plane with no more than $n/2-t(n)$ collinear points.

  Then, for $c=17$, the number of $c$-ordinary triangles spanned by $X$ is at least $\Omega(n^2 \cdot t(n))$. This bound
  is tight, apart from the multiplicative constant hidden in the $\Omega$-notation. 
\end{theorem}

Of course, if every line contains fewer than $n/2$ points, $X$ cannot be covered by two lines. It is very likely that
the above results remain true with a much smaller value of $c$, perhaps already with $c=3$. 

\smallskip

There are several interesting algorithmic questions related to the detection and enumeration of almost ordinary
triangles in a finite point set in the plane. We present two such results. 

\begin{proposition} \label{prop:reporting}
  Given a set $X$ of $n$ points in the plane and a positive integer $\tau$, all $\tau$-ordinary triangles
  spanned by $X$ can be reported in $O(n^3)$ time. The order of magnitude of this bound cannot be improved
  in the worst case.
\end{proposition}

The problem admits a faster solution if we need to find just one $\tau$-ordinary triangle or count the number of
such triangles.

\begin{theorem} \label{thm:detection}
  Given a set $X$ of $n$ points in the plane and a positive integer $\tau$, deciding whether $X$ spans
  a $\tau$-ordinary triangle and finding one, if it does,
  or counting all $\tau$-ordinary triangles, can be done in $O(n^\omega)$ time,
where $\omega < 2.372$ is the exponent of the best algorithm for matrix multiplication.
\end{theorem}

Our note is organized as follows. In Section~\ref{sec:prelim}, we collect some classical results and technical lemmas
needed for the proofs. All combinatorial results are proved in Section~\ref{sec:proofs}, while in
the last section, we establish Proposition~\ref{prop:reporting} and Theorem~\ref{thm:detection}.

\smallskip

\noindent\textbf{Acknowledgements.}
The authors thank an anonymous referee for essential corrections in the algorithmic section.
The authors also express their gratitude to Quentin Dubroff for his valuable remarks and,
in particular, for suggesting the formulation of Theorem~\ref{thm:n/2-t}.

The 2nd named author also wants to thank the Simons Laufer Mathematical Sciences Institute (Berkeley),
where part of the work was done during the Spring 2025 special program on Extremal Combinatorics.

\section{Preliminaries} \label{sec:prelim}

In 1951, Dirac~\cite{Dir51} conjectured that every set $X$ of $n$ noncollinear points in the plane has a point incident
to at least $\lfloor n/2\rfloor$ distinct lines connecting it to other elements of $X$. The same conjecture was made,
independently, by Motzkin. Although the conjecture turned out to be false, it is quite possible that it is true up to an
additive constant. Ten years later, Erd\H os~\cite{Erd61} formulated a weaker statement, usually referred to, as the
``weak Dirac conjecture:'' There is a constant $c'>0$ such that every $n$-element noncollinear set $X$  has a point
incident to at least $c'n$ distinct lines connecting it to other elements of $X$. In 1983, this weaker conjecture was
proved in two seminal papers by Beck~\cite{Be83} and, independently, by Szemer\'edi and Trotter~\cite{SzT83}. 

One of the initial observations~\cite{Be83} was the following.

\begin{lemma} \label{lem:beck2} {\rm (Beck~\cite[Obs.~15]{PW14}).}
  Let $P$ be a set of $n$ noncollinear points in the plane. Then at least half the lines
  spanned by $P$ contain at most $3$ points.
\end{lemma}

For the proofs of Theorems~\ref{thm:n-t} and~\ref{thm:n/2-t}, we need four other statements concerning incidences
between points and lines, due to Langer, Payne and Wood, and de Zeeuw. All of them are refinements of some technical lemmas
from Beck's paper~\cite{Be83}; see also~\cite{Han17,PP17}.

\begin{lemma} \label{lem:langer} {\rm (Langer~\cite[Prop.~11.3.1]{La03}).}
  Given $n$ points in the plane, with at most $2n/3$ collinear, the number of incidences
  between these points and their connecting lines is at least $n(n+3)/3$.
\end{lemma}

\begin{lemma} \label{lem:payne-wood} {\rm (Payne and Wood~\cite[Thm.~12]{PW14}).}
  Let $X$ be a set of $n$ points in the plane. If at most $m$ points of $X$ are collinear,
  then $X$ determines at least $\frac{1}{98} n(n-m)$ distinct connecting lines.
\end{lemma}

\begin{lemma} \label{lem:de-zeeuw1} {\rm (de Zeeuw~\cite[Thm.~2.1]{Zee18}).}
Given a set $X$ of $n$ points in the plane, at least one of the following holds:
\begin{enumerate}
\item Some connecting line is incident to at least $\gamma n$ points in $X$, where
  $\gamma= (6+\sqrt3)/9= 0.859\ldots$.
\item There are at least $n^2/9$ lines spanned by $X$.
  \end{enumerate}
\end{lemma}

\begin{lemma} \label{lem:de-zeeuw2} {\rm (de Zeeuw~\cite[Cor.~3.2]{Zee18}).}
  Let $X$ be a set of $n$ points in the plane, with at most $2n/3$ collinear
  and let $\L$ be the set of lines spanned by $X$.

  Then for any $k \geq 5$, the number of lines in $\L$ incident to at least $k$ points of $X$ is at most
  $4|\L|/(k-2)^2$.
\end{lemma}

For the algorithmic part, we use the following classical result on triangle detection and triangle counting.

\begin{lemma}\label{itai} {\rm (Itai and Rodeh~\cite{IR78}.)}
Deciding whether a graph contains a triangle and finding one if it does, or counting all triangles in a graph, can be
done in $O(n^\omega)$ time, where $\omega$ is the infimum of all values $\tau$ 
such that two $n \times n$ real matrices can be multiplied in $O(n^\tau)$ time.
\end{lemma}

The best known upper bound for $\omega$ is roughly $2.372$; see ~\cite{DWZ23,VXXZ24}.
The significance of the asymptotically fastest known algorithms for matrix multiplication is mainly  theoretical: most
of them are impractical~\cite[Ch.~10.2.4]{Man89}. On the other hand, no combinatorial algorithm for
triangle detection that runs in truly subcubic time is currently known~\cite{AFK+24}.

\section{Proofs of Theorems~\ref{thm:n-t} and~\ref{thm:n/2-t}}   \label{sec:proofs}

We start by giving two very similar constructions of point sets with many collinear points that cannot be covered by $k$
lines, for a fixed $k\ge 2$, but determine only a small number of $c$-ordinary triangles. The first construction proves
Proposition~\ref{prop:k1}. 

\subparagraph{Proof of Proposition~\ref{prop:k1}.}
Let $n \geq 2k+c$, and place $n-2k+1$ arbitrary points on the $x$-axis. Pick the remaining $2k-1$ points in such a way
that no $3$ of them are collinear, and none of the ${2k-1 \choose 2}$ lines spanned by them 
passes through any of the points selected on the $x$-axis; see Fig.~\ref{fig:f1}.
Let $X$ denote the resulting $n$-element set and $Y$ the set of $2k-1$ points not on the $x$-axis
($Y \subset X$).

\begin{figure}[htbp]
 \centering
 \includegraphics{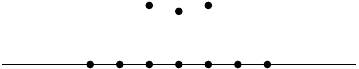}
 \caption{A counterexample to superlinearity (here $n=10$ and $k=2$).}
 \label{fig:f1}
\end{figure}

Every $c$-ordinary triangle must have at least two vertices in $Y$, and the number of such triangles is at most
\[ {2k-1 \choose 2} n < 2k^2 n. \]
Indeed, a $c$-ordinary triangle cannot have two vertices on the $x$-axis, because then that side  is not $c$-ordinary,
as $n-(2k-1) \geq c+1$. On the other hand, it is easy to see that the minimum number of lines that cover $X$ is $k+1$.
\qed
\smallskip

The following slight variation of the above construction shows that the order of magnitude of the lower bound in
Theorem~\ref{thm:n-t} cannot be improved, even if we only require that our point set cannot be covered by $k$ lines, for
any fixed $k\ge 2$. 

\begin{proposition} \label{prop:k2}
  Let $t(n) \leq 0.1 n$ be a function of $n$ with $\lim_{n \to \infty} t(n) =\infty$, and let $k, c \geq 2$ be fixed integers.

 Then there exists an $n$-element point set $X$ in the plane with at most $n-t(n)$ points on a line such that $X$ cannot
 be covered by $k$ lines, but the number of $c$-ordinary triangles spanned by $X$ is only $O(n \cdot t(n)).$ 
\end{proposition}

\subparagraph{Proof of Proposition~\ref{prop:k2}.}
Since $\lim_{n \to \infty} t(n) =\infty$, we may assume that $t(n) \geq 2k+c-2$, provided that $n$ is sufficiently large.
Select $n-t(n)$ points arbitrarily on the line $\ell_0: y=0$.
Pick another $t(n)-2k+3$ points on the line $\ell_1: y=1$, and the remaining
$2k-3$ points not on these lines and in such a way that the only collinear triples
are contained in one of the lines $\ell_i, i=0,1$; see Fig.~\ref{fig:f2}.
Let $X$ denote the resulting $n$-element set and $Y$ the set of $2k-3$ points not on
$\ell_0$ or $\ell_1$ ($Y \subset X$). Clearly, $X$ cannot be covered by $k$ lines.

\begin{figure}[htbp]
 \centering
 \includegraphics[scale=0.9]{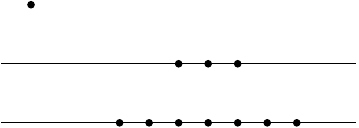}
 \caption{A set of $11$ points spanning $21$ ordinary triangles ($n=11$ and $k=2$).}
 \label{fig:f2}
\end{figure}

Every $c$-ordinary triangle must have at least one vertex in $Y$, and at most one vertex
on each of the lines $\ell_0$ and $\ell_1$.
Indeed, a $c$-ordinary triangle cannot have two vertices on $\ell_0$ or on $\ell_1,$ since these lines
are not $c$-ordinary, because $n -t(n) \geq c+1$ and $t(n)-2k+3 \geq c+1$. Therefore,
$X$ spans at most $(2k-3) \cdot n \cdot t(n) = O(n \cdot t(n))$ $c$-ordinary triangles.
\qed
\smallskip

Denote by $\ell(x,y)$ the line incident to $x$ and $y$, with $x,y \in X$, $x \neq y$.
The following auxiliary lemma is used in the proofs of Theorems~\ref{thm:n-t} and~\ref{thm:n/2-t}.

\begin{lemma} \label{lem:auxiliary}
Let $X$ be a set of $n$ points in the plane. Set $c=17$ and $\alpha=0.2$.
Let $\ell$ be a line incident to at least $\alpha n$ points in $X$, $L = X \cap \ell$ and
$Y= X \setminus L$.  Let $p,q \in Y$ such that $\ell(p,q)$ is $c$-ordinary with respect to $X$.

Then there exist at least $\alpha n/3$ points $r \in L$ such that $p,q,r$ form a $c$-ordinary triangle
in~$X$.
\end{lemma}
\begin{proof}
Define
\[ X_p = \{x \in L \ \colon \ |\ell(x,p) \cap X| \geq c+1\},
\text{ and }
X_q = \{x \in L \ \colon \ |\ell(x,q) \cap X| \geq c+1\}. \]
By definition, any point of $X_p$ determines a line in $\L$ that is incident to at least $c-1$ points
in $X \setminus (L \cup \{p\})$, and furthermore, these point sets are distinct for each $x \in X_p$.
Hence, we have $(c-1) |X_p| \leq n -|L|$, which gives
\[ |X_p| \leq \frac{n - |L|}{c-1} \leq \frac{1-\alpha}{c-1} n, \]
by the assumption of this first case.
Similarly, we have
\[ |X_q| \leq  \frac{1-\alpha}{c-1} n. \]
Recall that $\alpha = 0.2$ and $c=17$. A straightforward calculation implies that
\[ \alpha - \frac{2(1-\alpha)}{c-1} = \frac{\alpha}{2}, \]
whence there exist at least
\[ |L| - |X_p| - |X_q| -1 \geq \alpha n - \frac{2(1-\alpha)}{c-1} n -1 \geq \frac{\alpha}{3} n \]
points $r \in L$ such that $r \notin (X_p \cup X_q)$ and $p,q,r$ are noncollinear.
Furthermore,  $p,q,r$ form a $c$-ordinary triangle, as required.
\end{proof}

\subparagraph{Proof of Theorem~\ref{thm:n-t}.}
Assume that $n$ is sufficiently large: $n \geq n_0= 10^5$. Set
\begin{equation} \label{eq:c,alpha}
  c=17 \text{ and } \alpha= 0.2.
\end{equation}
Let $\L$ denote the set of lines determined by $X$.
As in the earlier existence proofs~\cite{Dub18,FMN+17}, we distinguish two cases.

\begin{enumerate} [(i)] \itemsep 1pt
\item There is a line $\ell \in \L$ such that $|\ell \cap X| \geq \alpha n$.
\item For every line $\ell \in \L,$ we have $|\ell \cap X| \leq \alpha n$.
\end{enumerate}

\smallskip
\emph{Case} \textbf{(i)}.
Let $L:= \{x \in \ell \cap X\}$ and $Y:= X \setminus L$.
Since $Y:= X \setminus L$ is noncollinear by assumption,
by the result of Kelly and Moser~\cite{KM58},
$Y$ determines $z \geq 3|Y|/7$ ordinary lines (in $Y$), say,
\[ \{p_j,q_j\}, j=1,\ldots,z. \]
Note that $|Y| \geq t(n)$, thus $z =\Omega(t(n))$.
Note also that every ordinary line in $Y$ is $3$-ordinary in~$X$.
By Lemma~\ref{lem:auxiliary}, each pair $\{p_j,q_j\}$ occurs in at least $|L|/3 = \Omega(n)$ $c$-ordinary triangles
with the third vertex on $\ell$. Since triangles from different pairs are distinct, this
completes the proof in Case (i). Indeed, we have $z |L|/3 = \Omega(n \cdot t(n))$, as required.

\smallskip
\emph{Case} \textbf{(ii)}. We sharpen the proof of Theorem~1.1 in~\cite{Dub18}, to conclude that
there are $\Omega(n^3) = \Omega(n \cdot t(n))$ $c$-ordinary triangles with $c=17$.
The main tools in the proof are incidence bounds for points-and-lines and averaging.

\smallskip
Let $X'$ be the set of points in $X$ that are incident to at least $33n/100$ lines
determined by $X$.

\begin{lemma} \label{lem:1/1000}
 We have $|X'| > n/1000$.
\end{lemma}
\begin{proof}
Recall that $\alpha=0.2$, hence $X$ satisfies the assumptions of Lemma~\ref{lem:langer}.
Let $\L$ denote the set of lines determined by $X$.
  Assume for contradiction that $|X'| \leq n/1000$. Let $I =I(X,\L)$ denote the set of incidences
  between $X$ and $\L$. We have
\begin{align*}
  |I| \leq |X'| \cdot (n-1) + |X \setminus X'| \cdot \frac{33n}{100}
  \leq \frac{n}{1000} \cdot n + \frac{999n}{1000} \cdot \frac{33n}{100} = \frac{33067n}{100000} n^2 < \frac13 n^2,
\end{align*}
contradicting Lemma~\ref{lem:langer}.
\end{proof}

\begin{lemma} \label{lem:0.27}
  Every point $p \in X'$ is incident to at least $0.27 n$ $c$-ordinary lines spanned by~$X$.
\end{lemma}
\begin{proof}
We claim that $p$ is incident to at most $n/c$ lines that are not $c$-ordinary.
Assume that this is not the case, \ie, there are more than $n/c$ lines,  each incident to
at least $c$ points in $X \setminus \{p\}$. Then there are at least
\[ \frac{n}{c} \cdot c + 1 = n+1 \]
points in $X$, a contradiction. Recall that $c=17$.
By the previous claim and Lemma~\ref{lem:1/1000}, it follows that $p$ is incident
to at least
\[ \frac{33n}{100} - \frac{n}{17} \geq 0.27 \, n\]
$c$-ordinary lines, as required.
\end{proof}

We can now finalize the proof of Theorem~\ref{thm:n-t}.
The assumption of Lemma~\ref{lem:de-zeeuw2} is satisfied and so
the number of lines in $\L$ that are not $c$-ordinary is at most
\begin{equation} \label{eq:2/81}
\frac{4}{(c-1)^2} \, |\L| =  \frac{4}{256} \, |\L| \leq \frac{1}{64} \, {n \choose 2} \leq
\frac{1}{128} \, n^2.
\end{equation}

Let $p \in X'$ be an arbitrary point in $X'$, and let $\L(p)$ denote the set of $c$-ordinary lines
incident to $p$. By Lemma~\ref{lem:0.27}, we have $|\L(p)| \geq 0.27 n$.
Let $X(p) \subseteq X$ denote the set of points in $X$  that are incident to at least one line in $\L(p)$.
Obviously, we have $|X(p)| \geq 0.27 n$. Let $\L(X(p))$ denote the set of lines spanned by $X(p)$.

Note that $\alpha =0.2 < \gamma \cdot 0.27$, where $\gamma$ is the constant in Lemma~\ref{lem:de-zeeuw1}.
Thus, $X(p)$ does not satisfy the first condition of Lemma~\ref{lem:de-zeeuw1}. Therefore, it must satisfy the
second condition:
\[ |\L(X(p))| \geq \frac{|X(p)|^2}{9} \geq \frac{n^2}{124}. \]
At most $n-1$ lines in $\L(X(p))$ are incident to $p$. Thus, at least
\[ \frac{n^2}{124} -n \geq \frac{n^2}{125} \]
lines in $\L(X(p))$ are not incident to $p$.
From~\eqref{eq:2/81}, it follows that $X(p)$ determines
\[ y(p) \geq \left( \frac{1}{125} - \frac{1}{128} \right) n^2  \geq \frac{1}{6000}\, n^2\]
$c$-ordinary lines that are not incident to $p$.
By construction, any pair of points in $X(p)$ that belong to one of these lines, together with $p$,
forms a $c$-ordinary triangle, so $p$ participates in $\Omega(n^2)$ $c$-ordinary triangles.
In view of Lemma~\ref{lem:1/1000}, the number of $c$-ordinary triangles with at least one vertex
in $X'$ is at least
\[ \Omega \left( \sum_{p \in X'} y(p) \right) = \Omega(n^3) = \Omega( n \cdot t(n)). \]

\smallskip

The tightness of this bound follows from Proposition~\ref{prop:k2}.
This concludes the proof of Theorem~\ref{thm:n-t}.
\qed

\subparagraph{Proof of Theorem~\ref{thm:n/2-t}.}
Set $c=17$ and $\alpha=0.2$.
Let $\L$ denote the set of lines determined by $X$.
As before, we distinguish two cases in the proof:

\begin{enumerate} [(i)] \itemsep 1pt
\item There is a line $\ell \in \L$ such that $|\ell \cap X| \geq \alpha n$.
\item For every line $\ell \in \L,$ we have $|\ell \cap X| \leq \alpha n$.
\end{enumerate}

\smallskip
\emph{Case} \textbf{(i)}. Let $\ell$ be one of the lines passing through the largest number of elements of $X$.
Let $L:= \{x \in \ell \cap X\}$ and $Y:= X \setminus L$.
Note that $|Y| \geq n/2 + t(n)$.
Let $\L'$ denote the set of lines spanned by $Y$.
Since at most $n/2 - t(n)$ points of $Y$ are collinear, by Lemma~\ref{lem:payne-wood},
$Y$ spans at least
\[ \frac{|Y| \cdot (|Y| -|L|)}{98} \geq \frac{(n/2 + t(n)) \cdot 2 t(n)}{98} = \Omega (n t(n)) \]
lines.

By assumption, $Y$ is noncollinear, so we can apply Beck's Lemma~\ref{lem:beck2}. We obtain that
at least half of the lines spanned by $Y$ contain at most $3$ points of $Y$.
Thus, among all lines spanned by $Y$, at least $\Omega (n t(n))$ are $4$-ordinary in $X$.
Letting $\L'_c$ denote the set of $c$-ordinary lines with respect to $X$ that are spanned by $Y$,
we have $|\L'_c| = \Omega (n t(n))$.
For each line $\ell' \in \L'_c$, fix two arbitrary points $p,q \in \ell'$.

By Lemma~\ref{lem:auxiliary}, each pair $p,q$ occurs in at least $|L|/3 = \Omega(n)$ $c$-ordinary triangles
with the third vertex on $\ell$. Since any two triangles that belong to different pairs are distinct, this
completes the proof in Case (i). Indeed, we have $|\L'_c| \cdot |L|/3 = \Omega(n^2 \cdot t(n))$, as required.

\smallskip
\emph{Case} \textbf{(ii)}.
Same proof as for case (ii) in Theorem~\ref{thm:n-t} applies.
The number of $c$-ordinary triangles found is $\Omega(n^3) = \Omega( n^2 \cdot t(n))$.
\smallskip

For the tightness, distribute the $n$ points on three parallel lines incident to
$n/2 -t(n)$, $n/2 -t(n)$, and $2 t(n)$ points, respectively, and observe that
any $c$-ordinary triangle must have a vertex on each of these lines. Therefore,
the number of $c$-ordinary triangles is at most
\[ \frac{n}{2} \cdot \frac{n}{2} \cdot 2 t(n)= \frac12 \cdot n^2 \cdot t(n), \]
as claimed.
\qed

\section{Algorithmic aspects}

In this section, we establish Proposition~\ref{prop:reporting} and Theorem~\ref{thm:detection}.
Let $\ell(a,b)$ denote the line spanned by $a,b \in X$.

\subparagraph{Proof of Proposition~\ref{prop:reporting}.}
Let $\L$ be the set of lines spanned by $X$. For every line $\ell \in \L$, let $I(\ell)$ denote the
number of incidences of $\ell$ with points in $X$.

In the first phase, the algorithm determines the counts $I(\ell)$ over all $\ell \in \L$
by using the standard point-line duality~\cite[Chap.~8]{BCKO08}, after computing the
dual line arrangement. In this setting, the count $I(\ell(a,b))$ is equal to the number of lines
incident to the intersection point of the dual lines $D(a)$ and $D(b)$.
Since the dual line arrangement can be computed in $O(n^2)$ time, the first phase takes $O(n^2)$ time.
Alternatively, computing $I(\ell)$ for all $\ell \in \L$ can be done by brute force in cubic time,
working in the primal setting.

In the second phase, the algorithm scans all triples $a,b,c$ of points in $X$; for each
such triple it determines whether the lines $\ell(a,b)$,  $\ell(b,c)$, and $\ell(c,a)$ are
$\tau$-ordinary and distinct. If so, the triangle $\Delta{abc}$ is $\tau$-ordinary and thus included in the output.
Duplicates can be easily filtered out. Since the number of triples is $O(n^3)$, the second phase,
and thus the entire algorithm, takes $O(n^3)$ time.

Note that for a set of $n$ points in general position, all ${n \choose 3}$ triangles are ordinary.
Similarly, under broad conditions as in Theorem~\ref{thm:n/2-t}, the number of $\tau$-ordinary
triangles is cubic when $\tau=17$. It follows that the algorithm is asymptotically optimal
in the worst case.
\qed

\subparagraph{Proof of Theorem~\ref{thm:detection}.}
Let $X=\{x_1,x_2,\ldots,x_n\}$. 
As in the proof of Proposition~\ref{prop:reporting}, in the first phase, the algorithm determines
the counts $I(\ell)$ over all $\ell \in \L$ in $O(n^2)$ time.
In the second phase, the algorithm constructs an undirected graph $G=(X,E)$ on vertex set $X$, where
two vertices are connected by an edge if and only if the connecting line is $\tau$-ordinary.
Since $|X|=n$, constructing $G$ takes $O(n^2)$ time based on the previous information.

In general, \ie, for $\tau \geq 3$, the set of $\tau$-ordinary triangles spanned by $X$ is only
a proper subset of the set of triangles in $G$; this is because three collinear points on a $\tau$-ordinary
line form a triangle in $G$ but do \emph{not} form a $\tau$-ordinary triangle. 

For counting, one can count all graph triangles with the algorithm of Itai and Rodeh
(cf. Lemma~\ref{itai}) and subtract ${s \choose 3}$ for each $\tau$-ordinary line containing $s$ points.
Detection can then be handled by checking whether the adjusted count is positive.

For actually finding a $\tau$-ordinary triangle, we show how to extract a non-collinear triangle
without increasing the running time. To this end, we adapt the the algorithm of Itai and Rodeh
for triangle detection in $G$ (cf. Lemma~\ref{itai}). Specifically we compute the following matrices:

\begin{enumerate}
\item $A=A[i,j]$ is the standard 0-1 adjacency matrix of $G$: $A[i,j]=1$ iff $\ell(x_i,x_j)$ is a
  $\tau$-ordinary line ($A[i,i]=0$);
\item $B=A^2$ is the integer matrix that records the number of paths of length $2$: \ie,
  for $i \neq j$, $B[i,j]$ is the number of paths $x_ix_kx_j$ between $x_i$ and $x_j$, $k \notin\{i,j\}$;
\item $C=C[i,j] =I(\ell(x_i,x_j))$ records the number of points in $X$ on the line $\ell(x_i,x_j)$,  for $i \neq j$;
\item $N=N[i,j]$ records the number of non-trivial paths of length $2$ between $x_i$ and $x_j$,  for $i \neq j$;
a path $x_ix_kx_j$ is \emph{non-trivial} if $x_k \notin \ell(x_i,x_j)$ and \emph{trivial} otherwise. Note that 
the number of two-edge trivial paths between $x_i$ and $x_j$ is $C[i,j]-2$, thus $N[i,j]= B[i,j] -(C[i,j] -2)$.
\end{enumerate}

Now a pair $x_i,x_j \in X$ determines a $c$-ordinary triangle $x_ix_jx_k$ iff $A[i,j]=1$ and $N[i,j] \geq 1$. 
Having computed the matrices $A$, $B$, $C$, and $N$ in $O(n^\omega)$ time, a $c$-ordinary triangle,
if it exists, can be found by looking at a pair $i,j$ with $A[i,j]=1$ and $N[i,j] \geq 1$
and identifying a witness $k$ for the latter inequality. The overall running time is $O(n^\omega)$.
\qed

\smallskip
Given $X$, finding the smallest $\tau$ for which there is a  $\tau$-ordinary triangle
spanned by $X$ is easily accomplished in $O(n^\omega \log{\tau})$ time: use the detection algorithm
from Theorem~\ref{thm:detection} for $\tau=2,4,8,\ldots$ until the answer is positive, followed by
binary search in the interval between the last negative answer and the first positive answer.

It would be interesting to know whether $\tau$-ordinary triangle detection can be done in quadratic time
or by a truly subcubic combinatorial algorithm.

\end{document}